\documentclass[reqno]{amsart}
\usepackage{amsmath,amsthm,amssymb}
\usepackage{enumerate}
\usepackage{hyperref}

\newtheorem{lemma}{Lemma}
\newtheorem{theorem}[lemma]{Theorem}
\newtheorem{claim}{Claim}[lemma]
\newtheorem*{definition}{Definition}
\newtheorem{conjecture}[lemma]{Conjecture}

\newcommand{\GF}[1]{\text{GF}(#1)}
\newcommand{\PGP}[3]{\text{PG}_{#3}(#1,#2)}
\newcommand{\PG}[2]{\text{PG}(#1,#2)}
\newcommand{\II}{\mathcal{I}}

\newcommand{\prob}[2][]{\mathbb{P}_{#1}\!\left(#2\right)}
\newcommand{\expect}[2][]{\mathbb{E}_{#1}\!\left[#2\right]}

\newcommand{\qbinom}[3]{\binom{#1}{#2}_{\!#3}}

\newcommand{\col}[1]{\textup{col}(#1)}

\newcommand{\term}[1]{\emph{#1}}

\author{Jorn van der Pol}
\address{University of Waterloo, Waterloo, ON, Canada}
\title{Random $\GF{q}$-representable matroids are not $(b,c)$-decomposable}

\begin{document}

\begin{abstract}
	We show that a random subset of the rank-$n$ projective geometry~$\PG{n-1}{q}$ is, with high probability, not $(b,c)$-decomposable: if $k$ is its colouring number, it does not admit a partition of its ground set into classes of size at most $ck$, every transversal of which is $b$-colourable. This generalises recent results by Abdolazimi, Karlin, Klein, and Oveis Gharan~\cite{AKKG} and by Leichter, Moseley, and Pruhs~\cite{LeichterMoseleyPruhs2022}, who showed that~$\PG{n-1}{2}$ is not $(1,c)$-decomposable, resp.\ not $(b,c)$-decomposable.
\end{abstract}

\maketitle

\section{Introduction}

A matroid $M=(E,\II)$, with ground set $E$ and independent sets $\II$ is \term{$k$-colourable} (also \term{$k$-coverable}) if its ground set can be partitioned into $k$ independent sets. The smallest such $k$ is called the \term{colouring number} of $M$, for which we write $\col{M}$. The colouring number of a matroid was studied by Edmonds~\cite{Edmonds1965}, who provided the following characterisation.
\begin{theorem}[Edmonds' Characterisation]
	$\col{M} = \max\limits_{X \subseteq E(M): r(X) > 0} \left\lceil\frac{|X|}{r(X)}\right\rceil$.
\end{theorem}

A $k$-colourable matroid $M$ is \term{$(1,c)$-decomposable} if its ground set can be partitioned into an arbitrary number of classes, each of which has cardinality at most $ck$, such that every transversal of the classes is independent. Equivalently, for such a matroid there exists a partition matroid $N$ on the same ground set, all of whose capacities are 1, such that every independent set in $N$ is independent in $M$ (in other words, the identity function is a \term{weak map} from $M$ to $N$). The notion of $(1,c)$-decomposition was introduced by B\'{e}rczi, Schwarcz, and Yamaguchi~\cite{BercziSchwarczYamaguchi2021}, who called it a $ck$-colourable partition reduction of $M$; the definition was subsequently extended to $(b,c)$-decomposability by Im, Moseley, and Pruhs~\cite{ImMoseleyPruhs2021}.
\begin{definition}
	A $k$-colourable matroid is \term{$(b,c)$-decomposable} if there is a partition $E = E_1 \cup E_2 \cup \ldots \cup E_{\ell}$, called a \term{$(b,c)$-decomposition}, such that
	\begin{enumerate}[(i)]
		\item $|E_i| \le ck$ for all $i \in [\ell]$, and
		\item every transversal $Y = \{y_1, y_2, \ldots, y_\ell\}$ with $y_i \!\in\! E_i$ for all $i \in [\ell]$ is \mbox{$b$-colourable}.
	\end{enumerate}
\end{definition}

B\'{e}rczi, Schwarcz, and Yamaguchi~\cite[Conjecture~1.10]{BercziSchwarczYamaguchi2021} conjectured that every matroid is $(1,2)$-decomposable. This was disproved by Abdolazimi, Karlin, Klein, and Oveis Gharan~\cite{AKKG}, who showed that, for sufficiently large~$n$, the rank-$n$ binary projective geometry $\PG{n-1,2}$ is not $(1,c)$-decomposable. Recently, Leichter, Moseley, and Pruhs~\cite{LeichterMoseleyPruhs2022} showed that the same matroid is not even $(b,c)$-decomposable, again provided that $n$ is sufficiently large.

\begin{theorem}[\cite{LeichterMoseleyPruhs2022}]\label{thm:LMP}
	For sufficiently large $n$, the rank-$n$ binary projective geometry $\PG{n-1}{2}$ admits no $(b,c)$-decomposition.
\end{theorem}

With minor modifications, their proof can be generalised to projective geometries over arbitrary finite fields.

\begin{theorem}\label{thm:LMP-q}
	Let $q\ge 2$ be a prime power. For sufficiently large $n$, the rank-$n$ $q$-ary projective geometry $\PG{n-1}{q}$ admits no $(b,c)$-decomposition.
\end{theorem}

The crux in the argument of~\cite{LeichterMoseleyPruhs2022} is an analysis of flats of large (depending on $b$) rank in $\PG{n-1}{q}$. On the one hand, the number of such flats grows rapidly as $n$ grows. On the other hand, if $\PG{n-1}{q}$ is $(b,c)$-decomposable, such flats have large colouring number, and therefore their number can be bounded from above. For large~$n$, this leads to a contradiction.

Let $\PGP{n-1}{q}{p}$ be the random binary matroid obtained by restricting the full projective geometry $\PG{n-1}{q}$ to a random subset $E$ whose elements are chosen independently with probability $p$.
The main contribution of the current note is that the contradiction leading to the result of~\cite{LeichterMoseleyPruhs2022} still holds with high probability in the random submatroid $\PGP{n-1}{q}{p}$, and thus that a random $\GF{q}$-representable matroid is not $(b,c)$-decomposable.
\begin{theorem}\label{thm:main}
	Let $q\ge 2$ be a prime power and let $p \in (0,1]$. Let $b,c \ge 1$. With high probability\protect\footnote{A sequence of events $\mathcal{E}_n$, indexed by $n$, occurs \term{with high probability} when $\lim\limits_{n\to\infty} \prob{\mathcal{E}_n} = 1$, or equivalently, $\lim\limits_{n\to\infty} \prob{\mathcal{E}_n^c} = 0$.}, $\PGP{n-1}{q}{p}$ is not $(b,c)$-decomposable.
\end{theorem}
Note that Theorems~\ref{thm:LMP} and~\ref{thm:LMP-q} can be recovered from Theorem~\ref{thm:main} by choosing $p=1$.

Finally, we compare the situation for random $\GF{q}$-representable matroids with the situation for random $n$-element matroids.
While Theorem~\ref{thm:main} with $p=1/2$ implies that the random $\GF{q}$-representable matroid is, with high probability, not $(b,c)$-decomposable, it is likely that a random matroid on~$n$ elements is decomposable: It is believed that almost every matroid is paving~\cite{CrapoRota1970,MayhewNewmanWelshWhittle2011,PendavinghVanderpol2015}, and B\'{e}rczi, Schwarz, and Yamaguchi~\cite{BercziSchwarczYamaguchi2021} showed that paving matroids of rank at least~2 are $(1,2)$-decomposable.
The following probabilistic version of the original conjecture by B\'{e}rczi, Schwarcz, and Yamaguchi still seems likely.
\begin{conjecture}
	With high probability, the random matroid on ground set $[n]$ is $(1,3/2)$-decomposable.
\end{conjecture}
This conjecture is weaker than the original conjecture because it allows for a small number of matroids that are not $(b,c)$-decomposable. At the same time, the conclusion for the remaining matroids is stronger, as $3/2 < 2$. The improved constant can be explained as follows.
The random matroid on a ground set with~$n$ elements has, with high probability, rank asymptotic to $n/2$~\cite[Corollary~2.3]{LowranceOxleySempleWelsh2013}; a paving matroid of rank $r \sim n/2$ is $k$-colourable~\cite[Lemma~3.5]{BercziSchwarczYamaguchi2021} and has a $\left\lceil\frac{rk}{r-1}\right\rceil$-colourable partition reduction~\cite[Theorem~1.6]{BercziSchwarczYamaguchi2021} for some $k \in \{2,3\}$. Finally, for $k \in \{2,3\}$ and $r$ sufficiently large we have $\left\lceil\frac{rk}{r-1}\right\rceil = k+1 \le \frac{3}{2}k$.

The remainder of this note is structured as follows. In Section~\ref{sec:preliminaries}, we introduce some of the tools we require. Then, in Section~\ref{sec:main} we prove Theorem~\ref{thm:main}.

\section{Preliminaries}
\label{sec:preliminaries}

We require two probabilistic bounds. The first estimates tail probabilities for nonnegative random variables, and the second is a concentration bound for sums of independent random variables.

\begin{lemma}[Markov inequality]
	Let $X$ be a nonnegative random variable and let $\mu = \expect{X}$. Then for all $x > 0$
	\begin{equation*}
		\prob{X \ge x} \le \frac{\mu}{x}.
	\end{equation*}
\end{lemma}

\begin{lemma}[Chernoff bound]
	Let $X_1, X_2, \ldots, X_N$ be independent random variables taking values in $\{0,1\}$. Let $X = \sum_{i=1}^N X_i$ and let $\mu = \expect{X}$. For all $0 \le \delta \le 1$,
	\begin{equation*}
		\prob{X\ge(1+\delta)\mu}\le\exp\left(-\frac{1}{3}\delta^2\mu\right)
	\end{equation*}
	and
	\begin{equation*}
		\prob{X\le(1-\delta)\mu}\le\exp\left(-\frac{1}{2}\delta^2\mu\right).
	\end{equation*}
\end{lemma}

We write $\qbinom{n}{d}{q}$ for $q$-binomial coefficients; that is, for $0\le d\le n$ and $q>1$
\begin{equation*}
	\qbinom{n}{d}{q} = \prod_{j=0}^{d-1} \frac{q^{n-j}-1}{q^{d-j}-1}.
\end{equation*}
When $q \ge 2$ is a prime power, $\qbinom{n}{d}{q}$ counts the number of rank-$d$ flats in~$\PG{n-1}{q}$. The following standard bounds are useful for estimating $q$-binomial coefficients.
\begin{lemma}\label{lemma:qbinom}
	$q^{d(n-d)} \le \qbinom{n}{d}{q} \le q^{d(n-d+1)}$ for all $\le d\le n$ and $q > 1$.
\end{lemma}

Throughout, we use $o(1)$ to denote a quantity that tends to~0 as the parameter~$n$ tends to infinity. We write $a = (1\pm b)c$ as shorthand for $(1-b)c \le a \le (1+b)c$.

\section{Proof of Theorem~\ref{thm:main}}
\label{sec:main}

\subsection{Random subsets of projective geometries}

We obtain a random submatroid of the projective geometry $\PG{n-1}{q}$ by retaining each of its elements independently with probability $p$. Writing $E_p$ for the resulting random set of points, we set $\PGP{n-1}{q}{p} = \PG{n-1}{q}|E_p$. This model of random $\GF{q}$-representable matroids was first studied by Kelly and Oxley~\cite{KellyOxley1982}, who obtained results about rank, connectivity, and critical exponent\protect\footnote{The critical exponent is also known as the critical number.} in this model.

\subsection{Size, rank, and colouring number of \texorpdfstring{$\PGP{n-1}{q}{p}$}{PGp(n-1,q)}}

In the following three lemmas, we analyse the size, rank, and colouring number of the random matroid $\PGP{n-1}{q}{p}$.

\begin{lemma}\label{lemma:PGP-size}
	Let $q\ge 2$ be a prime power and let $p \in (0,1]$. Let $\delta > 0$. With high probability, $|\PGP{n-1}{q}{p}| = (1\pm\delta)p\frac{q^n}{q-1}$. 
\end{lemma}

\begin{proof}
	This follows immediately from the Chernoff bound, upon observing that $|\PGP{n-1}{q}{p}|$ is the sum of $\frac{q^n-1}{q-1} \sim \frac{q^n}{q-1}$ independent indicator random variables, each with expected value~$p$.
\end{proof}

The next lemma was proved in~\cite[Theorem~4]{KellyOxley1982}, where it was shown that with high probability $\PGP{n-1}{q}{p}$ contains an $(n+1)$-circuit of $\PG{n-1}{q}$. Here, we provide an alternative proof.

\begin{lemma}\label{lemma:PGP-rank}
	Let $q\ge 2$ be a prime power and let $p \in (0,1]$. With high probability, $\PGP{n-1}{q}{p}$ is of full rank, that is $r(\PGP{n-1}{q}{p})=n$.
\end{lemma}

\begin{proof}
	If $\PGP{n-1}{q}{p}$ is not of full rank, then $E_p$ is contained in a hyperplane of $\PG{n-1}{q}$. By the union bound, this happens with probability at most
	\begin{equation*}
		\qbinom{n}{n-1}{q} (1-p)^{q^{n-1}} = \frac{q^n-1}{q-1} (1-p)^{q^{n-1}} = o(1). \qedhere
	\end{equation*}
\end{proof}

\begin{lemma}\label{lemma:PGP-colourable}
	Let $q\ge 2$ be a prime power and let $p \in (0,1]$. Let $\delta > 0$. With high probability, $\col{\PGP{n-1}{q}{p}} = (1\pm\delta)p\frac{q^n}{(q-1)n}$.
\end{lemma}

\begin{proof}
	We first prove the lower bound. By Lemma~\ref{lemma:PGP-size} and Lemma~\ref{lemma:PGP-rank}, with high probability, $\PGP{n-1}{q}{p}$ has at least $(1-\delta)p\frac{q^n-1}{q-1}$ points and has rank $n$. It follows from Edmonds' characterisation of the colouring number that
	\begin{equation*}
		\col{\PGP{n-1}{q}{p}} \ge \frac{(1-\delta)p\frac{q^n}{q-1}}{n}
	\end{equation*}
	with high probability.
	
	To prove the corresponding upper bound, it suffices to show that, with high probability,
	\begin{equation}\label{eq:small-flat}
		\text{for every flat $F$ of $\PG{n-1}{q}$:}\quad|E_p \cap F| \le (1+\delta)p\frac{q^n}{(q-1)n} r(E_p \cap F).
	\end{equation}
	We may assume that $n \ge \frac{1}{(1+\delta)p}$.
	
	Let $F$ be a flat of $\PG{n-1}{q}$ of rank $t > 0$. If $t \le n - 2\log_q n$, then
	\begin{equation*}
		|F| = \frac{q^t-1}{q-1}
			< \frac{q^{n-2\log_q n}}{q-1}
			= \frac{q^n}{(q-1)n^2}
			\le (1+\delta) p \frac{q^n}{(q-1)n},
	\end{equation*}
	so~\eqref{eq:small-flat} holds for all flats $F$ of rank at most $n-2\log_q n$.
	
	Next, let $t \ge n - 2\log_q n$. If $r(E_p \cap F) < t$, then $F$ contains a rank-$(t-1)$ flat $F'$ such that $E_p \cap F' = \emptyset$. This happens with probability at most
	\begin{equation*}
		\qbinom{t}{t-1}{q} (1-p)^{q^{t-1}}
		\le q^t(1-p)^{q^{t-1}}
	\end{equation*}
	By the Chernoff bound, the probability that $|E_p \cap F|$ is larger than $(1+\delta)p|F|$ is at most
	\begin{equation*}
		\exp\left(-\frac{1}{3}\delta^2p|F|\right)
		\le \exp\left(-\frac{1}{3}\delta^2pq^{t-1}\right).
	\end{equation*}
	It follows that for a flat $F$ of rank $t$, \eqref{eq:small-flat} fails with probability at most
	\begin{equation*}
		q^t(1-p)^{q^{t-1}}
			+\exp\left(-\frac{1}{3}\delta^2pq^{t-1}\right).
	\end{equation*}
	Summing over all flats of rank~$t$, it follows that~\eqref{eq:small-flat} fails with probability at most
	\begin{equation*}
		\sum_{t = n - 2\log_q n}^n \qbinom{n}{t}{q} 
			\left(
			q^t(1-p)^{q^{t-1}}
				+\exp\left(-\frac{1}{3}\delta^2pq^{t-1}\right)
			\right)
			= o(1).
	\end{equation*}
	Thus, \eqref{eq:small-flat} holds with high probability, which concludes the proof.
\end{proof}

\subsection{Proof of the main theorem}

We now prove Theorem~\ref{thm:main}, which we restate here for convenience.

\setcounter{lemma}{3}
\begin{theorem}
	Let $q\ge 2$ be a prime power and let $p \in (0,1]$. Let $b, c \ge 1$. With high probability, $\PGP{n-1}{q}{p}$ is not $(b,c)$-decomposable.
\end{theorem}

\begin{proof}
	Let $d = \lceil\log\log n\rceil$ and let $n_0$ be so large that
	\begin{equation*}
		d \ge 3,
		\qquad
		nq^{-d^2} > \frac{c^2(1+\delta)^2p^2}{(q-1)^2},
		\qquad\text{and}\qquad
		\frac{1}{2}p\frac{q^d-1}{(q-1)d} > b
	\end{equation*}
	for all $n \ge n_0$. We may assume that $n \ge n_0$.
	Let $k = (1+\delta)p\frac{q^n}{(q-1)n}$.
	For convenience, write $M = \PGP{n-1}{q}{p}$.
	
	We say that a rank-$d$ flat of $M$ is dense if it contains at least $\frac{1}{2} p \frac{q^d-1}{q-1}$ elements. Let $\mathcal{Z}_d$ be the set of dense rank-$d$ flats of $M$.
	
	\begin{claim}\label{claim:dense-flat}
		$\col{M|F} > b$ for all $F \in \mathcal{Z}_d$.
	\end{claim}
	\begin{proof}[Proof of claim]
		By Edmonds' characterisation and density,
		\begin{equation*}
			\col{M|F}
				\ge \frac{|F|}{d}
				\ge \frac{1}{2} p \frac{q^d-1}{(q-1)d}
				> b. \qedhere
		\end{equation*}
	\end{proof}
	
	Consider the following three properties:
	\begin{enumerate}[(i)]
		\item $M$ has at least $\frac{1}{2}p\frac{q^n-1}{q-1}$ elements.
		\item $M$ is $k$-colourable.
		\item $|\mathcal{Z}_d| \ge \frac{1}{2}\qbinom{n}{d}{q}$.
	\end{enumerate}
	We will show that each of these properties holds with high probability, and that if the three properties hold then $M$ is not $(b,c)$-decomposable.
	
	Property~(i) holds with high probability by Lemma~\ref{lemma:PGP-size}. Property~(ii) holds with high probability by Lemma~\ref{lemma:PGP-colourable}.
	
	\begin{claim}
		Property~(iii) holds with high probability.
	\end{claim}
	
	\begin{proof}[Proof of claim]
		Let $F$ be a flat of $\PG{n-1}{q}$. For $F$ to survive as a dense rank-$d$ flat of $M$, $|E_p \cap F|$ must be large while $r(E_p \cap F) = d$. The probability that $|F\cap E_p| < \frac{1}{2}p\frac{q^d-1}{q-1}$ is at most
		\begin{equation*}
			\exp\left(-\frac{1}{8}p\frac{q^d-1}{q-1}\right) = o(1)
		\end{equation*}
		by an application of the Chernoff bound, while the probability that $r(E_p \cap F) < d$ is at most
		\begin{equation*}
			\qbinom{d}{d-1}{q} (1-p)^{q^{d-1}} \le q^d(1-p)^{q^{d-1}} = o(1).
		\end{equation*}
		It follows that the expected number of flats of $F$ that do not survive as a dense rank-$d$ flat in $M$ is $o\left(\qbinom{n}{d}{q}\right)$. By the Markov inequality, the probability that more than $\frac{1}{2}\qbinom{n}{d}{q}$ rank-$d$ flats of $\PG{n-1}{q}$ do not survive as dense rank-$d$ flats in $M$ is at most
		\begin{equation*}
			o\left(\qbinom{n}{d}{q}\right)\!\Bigg/\frac{1}{2}\qbinom{n}{d}{q} = o(1),
		\end{equation*}
		and hence the probability that $|\mathcal{Z}_d| < \frac{1}{2}\qbinom{n}{d}{q}$ is~$o(1)$.
	\end{proof}
	
	Finally, we show that if (i)--(iii) hold, then $M$ is not $(b,c)$-decomposable --- the proof follows the argument used in~\cite{LeichterMoseleyPruhs2022}. Suppose that (i)--(iii) hold; for the sake of contradiction, assume that $M$ is $(b,c)$-decomposable, and let $\{E_1, E_2, \ldots, E_\ell\}$ be a $(b,c)$-decomposition.
	
	Let $F \in \mathcal{Z}_d$. By Claim~\ref{claim:dense-flat}, $M|F$ is not $b$-colourable. It follows that for every dense rank-$d$ flat of $M$ there is an index $i \in [\ell]$ such that $|F \cap E_i| \ge 2$.
	
	Every dense rank-$d$ flat of $M$ can therefore be specified by an element $i \in [\ell]$, a pair of elements in $E_i$, and $d-2$ elements outside of $E_i$ to complete a spanning subset of the flat. Thus, the number of dense rank-$d$ flats in $M$ is at most
	\begin{multline}\label{eq:contradiction1}
		\qquad|\mathcal{Z}_d| \le \ell\binom{ck}{2}\binom{q^n}{d-2}
			< n\frac{(ck)^2}{2}q^{n(d-2)} \\
			\le \frac{c^2(1+\delta)^2p^2}{2(q-1)^2n}q^{nd}
			\le \frac{1}{2}q^{nd-d^2} \le \frac{1}{2}\qbinom{n}{d}{q},\qquad
	\end{multline}
	where the penultimate inequality follows from $n \ge n_0$, and the final inequality follows from Lemma~\ref{lemma:qbinom}.
	
	Equation~\eqref{eq:contradiction1} contradicts Property~(iii), so $M$ is not $(b,c)$-decomposable.	
\end{proof}

\bibliographystyle{alpha}
\bibliography{bib}

\end{document}